\documentclass[12pt,bookmark=true]{article}% use "amsart" instead of "article" for AMSLaTeX format

 \RequirePackage{geometry}
\usepackage{graphicx}				% Use pdf, png, jpg, or eps0‡0¡ì with pdflatex; use eps in DVI mode

\usepackage{amssymb}
\usepackage{amsmath}
\usepackage{amsthm}
\usepackage{enumerate}
 \usepackage{lineno}

\newtheorem{theorem}{\indent Theorem}[section]
\newtheorem{corollary}[theorem]{\indent Corollary}

\newtheorem{lemma}[theorem]{\indent Lemma}
\newtheorem{remark}[theorem]{\indent Remark}%SetFonts

\begin{document}
\title{Insensitizing controls for a fourth order semi-linear parabolic equations}
\author{
Bo You\footnote{Email address: youb2013@xjtu.edu.cn}
 \\
{\small School of Mathematics and Statistics, Xi'an Jiaotong University} \\
\small Xi'an, 710049, P. R. China\\
Fang Li\footnote{Email address: fli@xidian.edu.cn}
	 \\
	{\small School of Mathematics and Statistics, Xidian University} \\
	\small Xi'an, 710071, P. R. China
	}

%\date{February 8, 2021}							% Activate to display a given date or no date
\maketitle

\begin{center}
\begin{abstract}
This paper is concerned with the existence of insensitizing controls for a fourth order semilinear parabolic equation. Here, the initial data is partially unknown, we would like to find controls such that a specific functional is insensitive for small perturbations of the initial data. In general, this kind of problems can be recast as a null controllability problem for a nonlinear cascade system. We will first prove a null controllability result for a linear problem by global Carleman estimates and dual arguments. Then, by virtue of Leray-Schauder's fixed points theorem, we conclude the null controllability for the cascade system in the semi-linear case.


\end{abstract}
\end{center}

\section{Introduction}
\def\theequation{1.\arabic{equation}}\makeatother
\setcounter{equation}{0}
Let $D\subset \mathbb{R}^n (n\geq 2)$ be a nonempty bounded connected open set with smooth boundary $\partial D,$ $T>0$ and $\omega\subset D$ is a small nonempty open subset which is usually referred to as a control domain. Denote by $Q=D\times (0,T),$ $\Sigma=\partial D\times (0,T),$ $Q_\omega=\omega\times(0,T).$  Let $\mathcal{O}\subset D$ be another open set which is the so-called observation set.

In this paper, we mainly consider the following semilinear fourth order parabolic equation with incomplete data:
\begin{equation}\label{eq1.1}
\begin{cases}
\frac{\partial y}{\partial t}+\Delta^2 y+a_0y+B_0\cdot\nabla y+B:\nabla^2 y+a_1\Delta y=F(y,\nabla y,\nabla^2y)+v\chi_\omega+f,\,\,\,\,\forall\,\,\,(x,t)\in Q,\\
y=\Delta y=0,\,\,\,\,\,\forall\,\,\,\,(x,t)\in \Sigma,\\
y(x,0)=y_0(x)+\tau\hat{y}_0(x),\,\,\,\,\forall\,\,\,\,x\in D.
\end{cases}
\end{equation}
Here, the functions $a_0,$ $a_1\in L^{\infty}(Q;\mathbb{R}),$ $B_0=(B_{01},B_{02},\cdots,B_{0n})\in L^{\infty}(Q;\mathbb{R}^n),$ $B=(B_{ij})_{n\times n}\in L^{\infty}(Q;\mathbb{R}^{n^2}),$ $f\in L^2(Q)$ is a given externally applied force, the function $F:\mathbb{R}\times\mathbb{R}^n\times\mathbb{R}^{n^2}\rightarrow\mathbb{R}$ is locally Lipschitz continuous, $\chi_\omega$ is the characteristic function of the set $\omega,$ $v\in L^2(Q_\omega)$ is a control function to be determined and the initial data $y(x,0)$ is partially unknown in the following sense:
\begin{enumerate}[(a)]
\item $y_0\in L^2(D)$ is known.
\item $\hat{y}_0\in L^2(D)$ is unknown with $\|\hat{y}_0\|_{L^2(D)}=1.$ 
\item $\tau$ is a small unknown real number.
\end{enumerate}
Let $y$ be the solution of problem \eqref{eq1.1} associated to $\tau$ and $v,$
we observe the solution of problem \eqref{eq1.1} via some functional $\Phi(y),$ which is called the sentinel. Here, the sentinel is defined by the square of the local $L^2$-norm of the state variable:
\begin{align}\label{eq1.2}
\Phi(y)=\frac{1}{2}\int_0^T\int_{\mathcal{O}}|y(x,t)|^2\,dxdt.
\end{align}
A control function $v$ is said to insensitize the functional $\Phi,$ if
\begin{align}\label{eq1.3}
\frac{\partial \Phi(y)}{\partial \tau}|_{\tau=0}=0,\,\,\,\,\forall\,\,\,\hat{y}_0\in L^2(D)\,\,\,\textit{with}\,\,\,\|\hat{y}_0\|_{L^2(D)}=1.
\end{align}
Thus, the insensitizing control problem is to seek for a control $v,$ such that the uncertainty in the initial data does not effect the measurement $\Phi$ at least at the first order.

To the best of our knowledge, this kind of insensitizing control problem was first considered by J. L. Lions in \cite{ljl2}. Later, in \cite{bo, ljl1}, the authors reformulated the insensitization problem with this kind of the sentinel $\Phi$ as a null controllability problem for a cascade system. Inspired by these works,  there have been many results concerning the existence of insensitizing controls in all kinds of different contexts. Initially, the existence of an approximate insensitizing controls (i.e., such that $\left|\partial_\tau\Phi(y)|_{\tau=0}\right|\leq\epsilon$) was proved in \cite{bo} for a semilinear heat system with $\mathcal{C}^1$ and globally Lipschitz nonlinearities. In \cite{tld}, the author proved for the linear heat equation that we cannot expect insensitivity to hold for all initial data, except when the control acts everywhere in $\Omega.$ Regarding the class of initial data that can be insensitized, the results in \cite{tld2} also give different results of positive and negative nature.  Later, the results in \cite{bo} was generalized in \cite{bo1, bo2, xy} to superlinear heat equation with nonlinear terms depending on the state and/or its gradient. In particular, there are some results about the existence of insensitizing controls for the parabolic equation with different boundary conditions. For example, the authors in \cite{zmm} proved the existence of insensitizing controls for the parabolic equations with dynamic boundary conditions. The existence of a local insensitizing control for the semilinear parabolic equations with nonlinear Fourier boundary conditions was established in \cite{bo4}. Moreover, the author in \cite{lx} proved the existence of insensitizing controls for the quasilinear parabolic equations. The existence of insensitizing controls for a phase field system was proved in \cite{cbmr}.Additionally, the authors studied the existence of insensitizing controls for the Navier-Stokes equation and the Boussinesq system (see \cite{cn4, cn3, cn2, gm}), the semilinear wave equations (see \cite{abf, tl}).  It is worthy to mention that the authors treated the case of a different type of sentinel consisting of the gradient of the solution of a parabolic equation in \cite{gs1, smc} and the case of the curl for the Stokes system in \cite{gs2}.

Adapting the computations in \cite{bo} to problem \eqref{eq1.1}-\eqref{eq1.3}, we conclude that the existence of a control $v$ such that \eqref{eq1.3} holds is equivalent to the existence of a control $v$ such that the solution $(y,q)$ of problem
\begin{equation}\label{eq1.4}
\begin{cases}
\frac{\partial y}{\partial t}+\Delta^2 y+a_0y+B_0\cdot\nabla y+B:\nabla^2 y+a_1\Delta y=F(y,\nabla y,\nabla^2y)+\chi_\omega v+f,\,\,\,\,\forall\,\,\,(x,t)\in Q,\\
-\frac{\partial q}{\partial t}+\Delta^2 q+a_0q-\nabla\cdot(B_0 q)+\sum_{i,j=1}^n\frac{\partial^2(B_{ij}q)}{\partial x_i\partial x_j}+\Delta(a_1q)=F_y(y,\nabla y,\nabla^2y)q\\
-\nabla\cdot(\nabla_pF(y,\nabla y,\nabla^2y)q)+\sum_{i,j=1}^n\frac{\partial^2(F_{r_{ij}}(y,\nabla y,\nabla^2y)q)}{\partial x_i\partial x_j}+\chi_{\mathcal{O}} y,\,\,\,\,\forall\,\,\,(x,t)\in Q,\\
y=\Delta y=0,\,\,\,q=\Delta q=0,\,\,\forall\,\,\,\,(x,t)\in\Sigma,\\
y(x,0)=y_0(x),\,\,q(x,T)=0,\,\,\forall\,\,\,\,x\in D
\end{cases}
\end{equation}
verifying 
\begin{align}\label{eq1.5}
q(x,0)=0, \,\,\,\,\forall\,\,\,x\in D,
\end{align}
where $p=\nabla y$ and $r_{ij}=\frac{\partial^2y}{\partial x_i\partial x_j}.$ 

In recent several years, there are some results about the controllability for fourth order parabolic equations in both one dimension (see \cite{cn5, cn7, cn6, ce1, ce, gp, hsv, lgm}) and the higher dimensions (see \cite{dji, gs, kk, lq, yh}). In particular, the approximate controllability and non-approximate controllability of higher order parabolic equations were studied in \cite{dji}. The author in \cite{yh} proved the null controllability of fourth order parabolic equations by using the ideas of \cite{lg}. It is worthy to mention that the Carleman inequality for a fourth order parabolic equation with $n\geq 2$ was first established in \cite{gs}. Later, the author in \cite{kk} proved the null controllability and the exact controllability to the trajectories at any time $T>0$ for the fourth order semi-linear parabolic equations with a control function acting at the interior.  The null controllability for fourth order stochastic parabolic equations was proved by duality arguments and a new global Carleman estimates in \cite{lq}. A unified weighted inequality for fourth-order partial differential operators was given in \cite {cy}. Moreover, they applied it to obtain the log-type stabilization result for the plate equation. Recently, we in \cite{yb} established the global Carleman estimates for the fourth order parabolic equations with low regularity terms subject to the homogeneous Dirichlet boundary conditions of $y$ as well as $\Delta y,$ and applied it to the null controllability. However, there is no results concerning the existence of insensitizing controls for fourth order semilinear parabolic equations. Since the insensitizing control problems describe some kind of stability of system \eqref{eq1.1} with respect to initial data, it is very meaningful to investigate the existence of insensitizing controls for problem \eqref{eq1.1}. 

The main objective of this paper is to study the insensitizing controls problem \eqref{eq1.1}-\eqref{eq1.3}. Inspired by the work in \cite{bo}, we conclude that the insensitizing controls problem \eqref{eq1.1}-\eqref{eq1.3} is equivalent to the partial null controllability of problem \eqref{eq1.4}.Thus, we need to establish an observation inequality for the adjoint problem \eqref{eq3.1} of the linearized system for problem \eqref{eq1.4} based on the duality arguments. But problem \eqref{eq1.4} is coupled, we will choose some suitable cut-off function and combine the global Carleman estimates to conclude the following inequality
\begin{align*}
\int_0^T\int_\mathcal{O}|\psi|^2e^{2s\alpha}\leq C\int_{Q_{\omega_1}}|\varphi|^2\,dxdt
\end{align*}
for some suitable subset $\omega_1$ of $D$ and weight function $\alpha,$ which will entails the desired observability inequality of problem \eqref{eq3.1}. Throughout this paper, we will always suppose that $\omega\cap\mathcal{O}\neq\emptyset$, which is a condition that has always been imposed as long as insensitizing controls are concerned. However, in \cite{tld1}, it has been proved that this is not a necessary condition for $\epsilon$-insensitizing controls for some linear parabolic equations (see also \cite{ms}).  Thus, we shall assume that $y_0\equiv 0$ which is a classical hypothesis in insensitization problems.

%In this subject, the first results were obtained in \cite{fce1} for the local exact controllability to the trajectories of the Navier-Stokes and Boussinesq systems when the closure of the control set $\omega$ intersects the boundary of $\Omega$. Later, this geometric assumption was removed for the Stokes system in \cite{cjm}, for the local null controllability of the Navier-Stokes system in \cite{cn1} and for the Boussinesq system in \cite{cn}. See also \cite{cjm1} for local null controllability of the 2-dimensional Navier-Stokes system in a torus with controls having one vanishing component.

The rest of this paper is organized as follows: in Section 2, we will recall some global Carleman estimates and prove a technique lemma. In Section 3, we will prove an observability inequality, which implies the existence of insensitizing controls for fourth order linear parabolic equations. Section 4 is devoted to the existence of insensitizing controls for the semilinear case. 
\section{\bf Preliminaries}
\def\theequation{2.\arabic{equation}}\makeatother
\setcounter{equation}{0}
In this section, we will recall the Carleman inequalities of fourth order parabolic equations and some lemmas used in the sequel. To this purpose, we first introduce the following weight functions.
\begin{lemma}(\cite{fav})\label{le2.1}
Let $\omega_0\subset\subset D$ be an arbitrary fixed subdomain of $D$ such that $\overline{\omega_0}\subset\omega.$ Then there exists a function $\eta\in\mathcal{C}^4(\overline{D})$ such that
\begin{align*}
\eta(x)>0,\,\,\,\,\forall\,\,\,x\in D;\,\,\eta(x)=0,\,\,\,\,\forall\,\,\,x\in\partial D,;\,\,|\nabla\eta(x)|>0,\,\,\,\,\textit{for\,\,\,all}\,\,\,x\in\overline{D\backslash\omega_0}.
\end{align*}
\end{lemma}
In order to state the global Carleman inequality, we define some weight functions:
\begin{align}
\label{eq2.1}&\alpha_0(x)=e^{\lambda(2\|\eta\|_{L^{\infty}(D)}+\eta(x))}-e^{4\lambda\|\eta\|_{L^{\infty}(D)}},\,\,\,\,\xi_0(x)=e^{\lambda(2\|\eta\|_{L^{\infty}(D)}+\eta(x))},\\
\label{eq2.2}&\alpha(x,t)=\frac{\alpha_0(x)}{\sqrt{t(T-t)}},\,\,\,\xi(x,t)=\frac{e^{\lambda(2\|\eta\|_{L^{\infty}(D)}+\eta(x))}}{\sqrt{t(T-t)}}.
\end{align}
Moreover, they possess the following properties:
\begin{align}\label{eq2.3}
\nabla\alpha=\nabla\xi=\lambda\xi\nabla\eta,\,\,
\xi^{-1}\leq \frac{T}{2},\,\,
|\alpha_t|+|\xi_t|\leq\frac{T}{2}\xi^3,\,\,\forall\,\,(x,t)\in Q.
\end{align}
\begin{lemma}\label{le2.2}(see \cite{gs, yb})
Assume that $z_0\in L^2(D),$ $g\in L^2(Q)$ and the functions $\alpha,$ $\xi$ are defined by \eqref{eq2.2}. Then there exists $\hat{\lambda}>0$ such that for an arbitrary $\lambda\geq \hat{\lambda},$ we can choose $s_0=s_0(\lambda)>0$ satisfying: there exists a constant $C=C(\lambda)>0$ independent of $s,$ such that the solution $z\in L^2(Q)$ to problem
\begin{equation}\label{eq2.4}
\begin{cases}
L^*z=-\frac{\partial z}{\partial t}+\Delta^2 z=g,\,\,\,\,\textit{in}\,\,\,Q,\\
z=\Delta z=0,\,\,\,\,\,\textit{on}\,\,\,\,\Sigma,\\
z(x,T)=z_0(x),\,\,\,\,\textit{in}\,\,\,\,D,
\end{cases}
\end{equation}
satisfies the following inequality:
\begin{enumerate}[(i)]
\item If $g\in L^2(Q),$ for any $\lambda\geq \hat{\lambda}$  and any $s\geq s_0(\lambda)(\sqrt{T}+T),$ one has
\begin{align*}
&\int_Qe^{2s\alpha}\left(s^6\lambda^8\xi^6|z|^2+s^4\lambda^6\xi^4|\nabla z|^2+s^3\lambda^4\xi^3|\Delta z|^2+s^2\lambda^4\xi^2|\nabla^2z|^2+s\lambda^2\xi|\nabla\Delta z|^2\right)\,dxdt\\
&+\int_Qe^{2s\alpha}\left(\frac{1}{s\xi}(|z_t|^2+|\Delta^2z|^2)\right)\,dxdt\\
&\leq C\left(\int_{Q_\omega}s^7\lambda^8\xi^7|z|^2e^{2s\alpha}\,dxdt+\int_Q|g|^2e^{2s\alpha}\,dxdt\right)
\end{align*}
\item If $g=g_0+\sum_{i=1}^n\frac{\partial g_i}{\partial x_i}-\sum_{i,j=1}^n\frac{\partial^2(B_{ij}y)}{\partial x_i\partial x_j}-\Delta(a_1y)$ with $g_i\in L^2(Q)$ for any $0\leq i\leq n,$ and $a_1\in L^{\infty}(Q;\mathbb{R}),$ $B=(B_{ij})_{n\times n}\in L^{\infty}(Q;\mathbb{R}^{n^2}),$ then 
\begin{align*}
&\int_Qe^{2s\alpha}\left(s^6\lambda^8\xi^6|z|^2+s^4\lambda^6\xi^4|\nabla z|^2+s^2\lambda^4\xi^2|\Delta z|^2+s^2\lambda^4\xi^2|\nabla^2z|^2\right)\,dxdt\\
\leq&C\int_Q\left(|g_0|^2+\sum_{i=1}^n(s\lambda\xi)^2|g_i|^2\right)e^{2s\alpha}\,dxdt+C\int_{Q_\omega}s^7\lambda^8\xi^7|z|^2e^{2s\alpha}\,dxdt
\end{align*}
for any $\lambda\geq \hat{\lambda}(\lambda)(1+\|B\|_{L^{\infty}(Q)}^{\frac{1}{2}}+\|a_1\|_{L^{\infty}(Q)}^{\frac{1}{2}})$  and any $s\geq s_0(\sqrt{T}+T).$ 
\end{enumerate}
\end{lemma}
In what follows, we also prove the following technical lemma, which will be used to establish an observability inequality.
\begin{lemma}\label{le2.3}
Let the functions $\alpha_0,$ $\xi_0,$ $\alpha,$ $\xi$ be defined by \eqref{eq2.1} and \eqref{eq2.2}, denote by $m_0=\min\limits_{x\in D}\alpha_0(x),$ $M_0=\max\limits_{x\in D}\alpha_0(x)<0,$ $n_0=\min\limits_{x\in D}\xi_0(x)>0$ and $N_0=\max\limits_{x\in D}\xi_0(x).$ Then the following conclusions hold:
\begin{enumerate}[(1)]
\item For any $s\geq \frac{4T}{|M_0|},$ we have
\begin{align*}
s^{16}\xi^{16}e^{2s\alpha}\leq 2^{48}\left(\frac{N_0}{M_0e}\right)^{16}
\end{align*}
for any $(x,t)\in Q.$
\item  For any $s\geq 0,$ we have
\begin{align*}
s^6\xi^6e^{2s\alpha}\geq A_se^{-\frac{M_s}{\sqrt{t}}}
\end{align*}
for any $(x,t)\in\Omega\times (0,\frac{T}{2}),$ where
\begin{align*}
A_s=\frac{(2sn_0)^6}{T^6}e^{\frac{-2|m_0|s}{T}},\,\,\,\,
M_s=\frac{2|m_0|s}{\sqrt{T}}.
\end{align*}
\item  For any $s\geq 0,$ we have
\begin{align*}
\xi^{-6}e^{-2s\alpha}\leq (2n_0)^{-6}T^6e^{\frac{8|m_0|s}{\sqrt{3}T}}
\end{align*}
for any $(x,t)\in\Omega\times (\frac{T}{4},\frac{3T}{4}).$
\end{enumerate}
\end{lemma}
\begin{proof}
\begin{enumerate}[(i)]
\item Let $\alpha_0,$ $\alpha,$ $N_0$ and $M_0$ be as in the statement. Then we have
\begin{align*}
s^{16}\xi^{16}e^{2s\alpha}\leq  (sN_0)^{16}e^{-\frac{2|M_0|s}{\sqrt{t(T-t)}}}t^{-8}(T-t)^{-8}=f_s(t)=\frac{1}{g_s(t)}
\end{align*}
for any $s>0$ and any $t\in (0,T).$ In what follows, we will give a lower bound of $g_s(t)$ on $(0,T).$ Thanks to
\begin{align*}
g'_s(t)=\frac{1}{(sN_0)^{16}}e^{\frac{2|M_0|s}{\sqrt{t(T-t)}}}t^{\frac{13}{2}}(T-t)^{\frac{13}{2}}(T-2t)\left\{8\sqrt{t(T-t)}-|M_0|s\right\},
\end{align*}
which implies that for any $s\geq \frac{4T}{|M_0|},$ the function $g_s$ is strictly decreasing in $(0,\frac{T}{2})$ and strictly increasing in $(\frac{T}{2},T).$ Thus, we have
\begin{align*}
f_s(t)\leq f_s(\frac{T}{2})=2^{16}T^{-16}N_0^{16}G(s)
\end{align*}
for any $t\in (0,T)$ with $G(s)=s^{16}e^{-\frac{4|M_0|s}{T}}.$ Thanks to
\begin{align*}
G'(s)=4s^{15}e^{-\frac{4|M_0|s}{T}}(4-\frac{|M_0|s}{T}),
\end{align*}
which entails that the function $G(s)$ is strictly decreases in $(\frac{4T}{|M_0|},+\infty).$ Thus, for every $s\geq \frac{4T}{|M_0|},$ we have
\begin{align*}
s^{16}\xi^{16}e^{2s\alpha}\leq 2^{16}T^{-16}N_0^{16}G(\frac{4T}{|M_0|})=2^{48}e^{-16}\left(\frac{N_0}{M_0}\right)^{16}
\end{align*}
for any $(x,t)\in Q.$

\item First of all, notice that
\begin{align*}
s^6\xi^6e^{2s\alpha}\geq (sn_0)^6e^{-\frac{2|m_0|s}{\sqrt{Tt}}}t^{-3}(T-t)^{-3}e^{-\frac{2|m_0|s}{\sqrt{t(T-t)}}+\frac{2|m_0|s}{\sqrt{Tt}}}
\end{align*}
and for any $s\geq 0,$
\begin{align*}
-\frac{2|m_0|s}{\sqrt{t(T-t)}}+\frac{2|m_0|s}{\sqrt{Tt}}=&-\frac{2|m_0|s\sqrt{t}}{\sqrt{T(T-t)}(\sqrt{T}+\sqrt{T-t})}\\
\geq&-\frac{2|m_0|s}{T}
\end{align*}
for any $t\in (0,\frac{T}{2}).$ Therefore, for any $s\geq 0,$ we obtain
\begin{align*}
s^6\xi^6e^{2s\alpha}\geq &(sn_0)^6e^{-\frac{2|m_0|s}{\sqrt{Tt}}}t^{-3}(T-t)^{-3}e^{-\frac{2|m_0|s}{T}}\\
\geq &\frac{(2sn_0)^6}{T^6}e^{-\frac{2|m_0|s}{T}}e^{-\frac{2|m_0|s}{\sqrt{Tt}}}
\end{align*}
for any $t\in (0,\frac{T}{2}).$

\item Thanks to 
\begin{align*}
\frac{1}{\sqrt{t(T-t)}}\leq \frac{4}{\sqrt{3}T}
\end{align*}
for any $t\in (\frac{T}{4},\frac{3T}{4}),$ we obtain
\begin{align*}
\xi^{-6}e^{-2s\alpha}\leq n_0^{-6}e^{\frac{2|m_0|s}{\sqrt{t(T-t)}}}2^{-6}T^6\leq (2n_0)^{-6}T^6e^{\frac{8|m_0|s}{\sqrt{3}T}}
\end{align*}
for any $(x,t)\in\Omega\times (\frac{T}{4},\frac{3T}{4}).$
\end{enumerate}
\end{proof}
\section{\bf The linear case}
\def\theequation{3.\arabic{equation}}\makeatother
\setcounter{equation}{0}
In this section, we will always assume that $F\equiv 0$ and prove the existence of an insensitizing control of problem \eqref{eq1.1} such that \eqref{eq1.3} holds. To start with, we introduce the adjoint problem of the linearized system of problem \eqref{eq1.4}:
\begin{equation}\label{eq3.1}
\begin{cases}
-\frac{\partial \psi}{\partial t}+\Delta^2 \psi+a_0\psi-\nabla\cdot(B_0 \psi)+\sum_{i,j=1}^n\frac{\partial^2(B_{ij}\psi)}{\partial x_i\partial x_j}+\Delta(a_1\psi)=\chi_{\mathcal{O}} \varphi,\,\,\,\,\forall\,\,\,(x,t)\in Q,\\
\frac{\partial \varphi}{\partial t}+\Delta^2 \varphi+a_0\varphi+B_0\cdot\nabla\varphi+B:\nabla^2\varphi+a_1\Delta\varphi=0,\,\,\,\,\forall\,\,\,(x,t)\in Q,\\
\psi=\Delta \psi=0,\,\,\,\varphi=\Delta \varphi=0,\,\,\,\,\,\forall\,\,\,\,(x,t)\in\Sigma,\\
\psi(x,T)=0,\,\,\,\varphi(x,0)=\varphi_0(x),\,\,\,\,\forall\,\,\,\,x\in D.
\end{cases}
\end{equation}
From the regularity of fourth order parabolic equations, we conclude that for any $\varphi_0\in L^2(D),$ there exists a unique solution of problem \eqref{eq3.1} satisfying
\begin{align*}
\varphi,\psi\in X=L^2(0,T; H_0^1(D)\cap H^2(D))\cap H^1(0,T;(H^2(D))^*).
\end{align*}
In what follows, we will establish an observability inequality of problem \eqref{eq3.1}, which used to obtain the existence of an insensitizing  control such that the solution of problem \eqref{eq1.1} verifying \eqref{eq1.3} in the linear case.
\begin{theorem}\label{th3.1}
Assume that $\omega\cap \mathcal{O}\neq\emptyset.$ Then there exist two positive constants $M$ and $H,$ such that for any $\varphi_0\in L^2(D),$ the corresponding solution $(\psi,\varphi)$ of problem \eqref{eq3.1} with initial data $(0,\varphi_0)$ satisfies
\begin{align}\label{eq3.0}
\int_Qe^{-\frac{M}{\sqrt{t}}}|\psi|^2\,dxdt\leq H\int_{Q_\omega}|\psi|^2\,dxdt.
\end{align}
More precisely, $M=\frac{2|m_0|s}{\sqrt{T}}$ and
\begin{align*}
H=C2^{48}\left(\frac{N_0}{M_0e}\right)^{16}+C2^{42}\left(\frac{N_0}{M_0e}\right)^{16}e^{2\beta T+\frac{8|m_0|s}{\sqrt{3}T}}n_0^{-6}T^6
\end{align*}
for any $s\geq \frac{4T}{|M_0|},$ where $C=C(D,\omega,\mathcal{O}).$
\end{theorem}
\begin{proof}
Let $\omega_1$ and $\omega_2$ be two open subsets such that $\omega_1\subset\subset\omega_2\subset\subset\omega\cap\mathcal{O}.$ Applying Lemma \ref{le2.2} to the second equation of problem \eqref{eq3.1} with $g=-a_0\varphi-B_0\cdot\nabla\varphi-B:\nabla^2 \varphi-a_1\Delta\varphi$ and $\omega=\omega_1,$ we conclude that there exists a positive constant $\hat{\lambda},$ such that for any $\lambda\geq \hat{\lambda},$ we can choose $s_0=s_0(\lambda)$ satisfying: there exist a positive constant $C_1=C_1(D,\omega_1),$ such that
\begin{align*}
&\int_Qe^{2s\alpha}\left(s^6\lambda^8\xi^6|\varphi|^2+s^4\lambda^6\xi^4|\nabla \varphi|^2+s^3\lambda^4\xi^3|\Delta \varphi|^2+s^2\lambda^4\xi^2|\nabla^2\varphi|^2+s\lambda^2\xi|\nabla\Delta \varphi|^2\right)\,dxdt\\
\leq&C_1\int_Q\left(|a_0|^2|\varphi|^2+|B_0|^2|\nabla\varphi|^2+|B|^2|\nabla^2\varphi|^2+|a_1|^2|\Delta\varphi|^2\right)e^{2s\alpha}\,dxdt\\
&+C_1\int_{Q_{\omega_1}}s^7\lambda^8\xi^7|\varphi|^2e^{2s\alpha}\,dxdt
\end{align*}
for any $s\geq s_0(\lambda)(T+\sqrt{T}),$ which implies that
\begin{align}\label{eq3.2}
\nonumber&\int_Qe^{2s\alpha}\left(s^6\lambda^8\xi^6|\varphi|^2+s^4\lambda^6\xi^4|\nabla \varphi|^2+s^3\lambda^4\xi^3|\Delta \varphi|^2+s^2\lambda^4\xi^2|\nabla^2\varphi|^2+s\lambda^2\xi|\nabla\Delta \varphi|^2\right)\,dxdt\\
\leq&C\int_{Q_{\omega_1}}s^7\lambda^8\xi^7|\varphi|^2e^{2s\alpha}\,dxdt
\end{align}
for any $\lambda\geq \hat{\lambda}(1+\|a_0\|_{L^{\infty}(Q)}^{\frac{1}{4}}+\|B_0\|_{L^{\infty}(Q)}^{\frac{1}{3}}+\|B\|_{L^{\infty}(Q)}^{\frac{1}{2}}+\|a_1\|_{L^{\infty}(Q)}^{\frac{1}{2}})$ and any $s\geq s_0(\lambda)(T+\sqrt{T}).$

Employing again Lemma \ref{le2.2} to the first equation of problem \eqref{eq3.1} with $g=-a_0\psi+\nabla\cdot(B_0 \psi)-\sum_{i,j=1}^n\frac{\partial^2(B_{ij}\psi)}{\partial x_i\partial x_j}-\Delta(a_1\psi)$ and $\omega=\omega_2,$ we deduce that there exists a positive constant $\hat{\lambda},$ such that for any $\lambda\geq \hat{\lambda},$ we can choose $s_0=s_0(\lambda)$ satisfying: there exist a positive constant $C_2=C_2(D,\omega_2),$ such that
\begin{align*}
&\int_Qe^{2s\alpha}\left(s^6\lambda^8\xi^6|\psi|^2+s^4\lambda^6\xi^4|\nabla\psi|^2+s^2\lambda^4\xi^2|\Delta\psi|^2+s^2\lambda^4\xi^2|\nabla^2\psi|^2\right)\,dxdt\\
&\leq C_2\left(\int_{Q_{\omega_2}}s^7\lambda^8\xi^7|\psi|^2e^{2s\alpha}\,dxdt+\int_Q\left(|a_0|^2|\psi|^2+(s\lambda\xi)^2|B_0|^2|\psi|^2+\chi_{\mathcal{O}} |\varphi|^2\right)e^{2s\alpha}\,dxdt\right),
\end{align*}
for any $\lambda\geq \hat{\lambda}(1+\|B\|_{L^{\infty}(Q)}^{\frac{1}{2}}+\|a_1\|_{L^{\infty}(Q)}^{\frac{1}{2}})$  and any $s\geq s_0(\lambda)(\sqrt{T}+T),$ which entails that
\begin{align}\label{eq3.3}
\nonumber&\int_Qe^{2s\alpha}\left(s^6\lambda^8\xi^6|\psi|^2+s^4\lambda^6\xi^4|\nabla\psi|^2+s^2\lambda^4\xi^2|\Delta\psi|^2+s^2\lambda^4\xi^2|\nabla^2\psi|^2\right)\,dxdt\\
&\leq C\left(\int_{Q_{\omega_2}}s^7\lambda^8\xi^7|\psi|^2e^{2s\alpha}\,dxdt+\int_0^T\int_{\mathcal{O}}|\varphi|^2e^{2s\alpha}\,dxdt\right)
\end{align}
for any $\lambda\geq \hat{\lambda}(\lambda)(1+\|a_0\|_{L^{\infty}(Q)}^{\frac{1}{4}}+\|B_0\|_{L^{\infty}(Q)}^{\frac{1}{3}}+\|B\|_{L^{\infty}(Q)}^{\frac{1}{2}}+\|a_1\|_{L^{\infty}(Q)}^{\frac{1}{2}})$ and any $s\geq s_0(T+\sqrt{T}).$
In what follows, we will prove an inequality which bounds $\varphi$ with respect to $\psi.$ Let $\theta_1\in\mathcal{C}_0^{\infty}(\omega_2)$ be a cut-off function such that
\begin{align}\label{eq3.4}
0\leq\theta_1\leq 1,\,\,\,\textit{in}\,\,\,\omega_2;\,\,\,\theta_1\equiv 1,\,\,\,\forall\,\,\,x\in\omega_1.
\end{align}
Define 
\begin{align*}
u=s^7\lambda^8\xi^7e^{2s\alpha}
\end{align*}
 for any $\lambda\geq \hat{\lambda}(1+\|a_0\|_{L^{\infty}(Q)}^{\frac{1}{4}}+\|B_0\|_{L^{\infty}(Q)}^{\frac{1}{3}}+\|B\|_{L^{\infty}(Q)}^{\frac{1}{2}}+\|a_1\|_{L^{\infty}(Q)}^{\frac{1}{2}})$ and any $s\geq s_0(\lambda)(T+\sqrt{T}).$ Multiplying the first equation of problem \eqref{eq3.1} by $u\varphi\theta_1$ and integrating by parts, we obtain
 \begin{align}\label{eq3.5}
\nonumber&\int_0^T\int_{\mathcal{O}} s^7\lambda^8\xi^7|\varphi|^2e^{2s\alpha}\theta_1\,dxdt=\int_Q\psi u_t\theta_1\varphi+4\nabla\Delta\varphi\cdot\nabla(u\theta_1)\psi+2\Delta\varphi\Delta(u\theta_1)\psi\,dxdt\\
\nonumber&+\int_Q4\nabla^2\varphi:\nabla^2(u\theta_1)\psi+4\nabla\varphi\cdot\nabla\Delta(u\theta_1)\psi+\Delta^2(u\theta_1)\varphi\psi+B_0\cdot\nabla(u\theta_1)\varphi\psi\,dxdt\\
\nonumber&+\int_Q\sum_{i,j=1}^n\left(B_{ij}\frac{\partial\varphi}{\partial x_i}\frac{\partial (u\theta_1)}{\partial x_j}\psi+B_{ij}\frac{\partial\varphi}{\partial x_j}\frac{\partial (u\theta_1)}{\partial x_i}\psi+B_{ij}\frac{\partial^2 (u\theta_1)}{\partial x_i\partial x_j}\psi\varphi\right)\,dxdt\\
&+\int_Q\left(2a_1\nabla\varphi\cdot\nabla(u\theta_1)\psi+a_1\varphi\psi\Delta(u\theta_1)\right)\,dxdt=:\sum_{i=1}^{12}I_i.
\end{align}
 In what follows, let $C$ be a positive constant depending only on $D,$ $\omega_1$ and $\omega_2,$ which may change from one line to another, we will estimate each $I_i$ in inequality \eqref{eq3.5} for $1\leq i\leq 12$ by H\"{o}lder's inequality, Young's inequality along with inequality \eqref{eq3.2}.
 
 To begin with, we conclude from the properties of weight functions \eqref{eq2.3} that
 \begin{align*}
|u_t|\leq &Cs^{10}\lambda^8\xi^{10}e^{2s\alpha},\\
|\nabla^k(u\theta_1)|\leq&C(s^{7+k}\lambda^{8+k}\xi^{7+k})e^{2s\alpha}\chi_{\omega_2},\,\,\,\forall\,\,k\in\mathbb{Z}_+
 \end{align*}
 for any $\lambda\geq\hat{\lambda}$ and any $s\geq s_0(1+\sqrt{T}+T).$ Thus, we obtain
 \begin{align}\label{eq3.6}
  \nonumber|I_1|\leq &C\int_Qs^{10}\lambda^8\xi^{10}|\varphi||\psi|\theta_1e^{2s\alpha}\,dxdt\\
 \leq&\frac{1}{14}\int_Qs^7\lambda^8\xi^7|\varphi|^2\theta_1e^{2s\alpha}\,dxdt+C\int_{Q_{\omega_2}}s^{13}\lambda^8\xi^{13}|\psi|^2e^{2s\alpha}\,dxdt,
  \end{align}
  
  \begin{align}\label{eq3.7}
  \nonumber |I_2|+|I_3|+|I_4|\leq&C\int_{Q_{\omega_2}}\left(s^8\lambda^9\xi^8|\nabla\Delta\varphi||\psi|+s^9\lambda^{10}\xi^9|\Delta\varphi||\psi|\right)e^{2s\alpha}\,dxdt\\
 \nonumber\leq&C\left(\int_Q\left(s\lambda^2\xi|\nabla\Delta \varphi|^2+s^3\lambda^4\xi^3|\Delta \varphi|^2\right)e^{2s\alpha}\,dxdt\right)^{\frac{1}{2}}\left(\int_{Q_{\omega_2}}s^{15}\lambda^{16}\xi^{15}|\psi|^2e^{2s\alpha}\,dxdt\right)^{\frac{1}{2}}\\
  \leq&\frac{1}{14}\int_Qs^7\lambda^8\xi^7|\varphi|^2\theta_1e^{2s\alpha}\,dxdt+C\int_{Q_{\omega_2}}s^{15}\lambda^{16}\xi^{15}|\psi|^2e^{2s\alpha}\,dxdt,
  \end{align}

  \begin{align}\label{eq3.8}
  \nonumber|I_5|+|I_6|\leq&C\int_{Q_{\omega_2}}\left(s^{10}\lambda^{11}\xi^{10}|\nabla\varphi||\psi|+s^{11}\lambda^{12}\xi^{11}|\varphi||\psi|\right)e^{2s\alpha}\,dxdt\\
 \nonumber\leq&C\left(\int_Q\left(s^4\lambda^6\xi^4|\nabla \varphi|^2+s^6\lambda^8\xi^6|\varphi|^2\right)e^{2s\alpha}\,dxdt\right)^{\frac{1}{2}}\left(\int_{Q_{\omega_2}}s^{16}\lambda^{16}\xi^{16}|\psi|^2e^{2s\alpha}\,dxdt\right)^{\frac{1}{2}}\\
  \leq&\frac{1}{14}\int_Qs^7\lambda^8\xi^7|\varphi|^2\theta_1e^{2s\alpha}\,dxdt+C\int_{Q_{\omega_2}}s^{16}\lambda^{16}\xi^{16}|\psi|^2e^{2s\alpha}\,dxdt,
  \end{align}

  \begin{align}\label{eq3.9}
  \nonumber|I_7|\leq&C\int_{Q_{\omega_2}}|B_0|s^8\lambda^9\xi^8|\varphi||\psi|e^{2s\alpha}\,dxdt\\
 \nonumber\leq&C\|B_0\|_{L^{\infty}(Q)}\left(\int_Qs^6\lambda^8\xi^6|\varphi|^2e^{2s\alpha}\,dxdt\right)^{\frac{1}{2}}\left(\int_{Q_{\omega_2}}s^{10}\lambda^{10}\xi^{10}|\psi|^2e^{2s\alpha}\,dxdt\right)^{\frac{1}{2}}\\
  \leq&\frac{1}{14}\int_Qs^7\lambda^8\xi^7|\varphi|^2\theta_1e^{2s\alpha}\,dxdt+C\|B_0\|_{L^{\infty}(Q)}^2\int_{Q_{\omega_2}}s^{10}\lambda^{10}\xi^{10}|\psi|^2e^{2s\alpha}\,dxdt,  \end{align}

  \begin{align}\label{eq3.10}
   \nonumber|I_8|+|I_9|\leq&C\int_{Q_{\omega_2}}|B|s^8\lambda^9\xi^8|\nabla\varphi||\psi|e^{2s\alpha}\,dxdt\\
 \nonumber\leq&C\|B\|_{L^{\infty}(Q)}\left(\int_Qs^4\lambda^6\xi^4|\nabla\varphi|^2e^{2s\alpha}\,dxdt\right)^{\frac{1}{2}}\left(\int_{Q_{\omega_2}}s^{12}\lambda^{12}\xi^{12}|\psi|^2e^{2s\alpha}\,dxdt\right)^{\frac{1}{2}}\\
  \leq&\frac{1}{14}\int_Qs^7\lambda^8\xi^7|\varphi|^2\theta_1e^{2s\alpha}\,dxdt+C\|B\|_{L^{\infty}(Q)}^2\int_{Q_{\omega_2}}s^{12}\lambda^{12}\xi^{12}|\psi|^2e^{2s\alpha}\,dxdt,    
  \end{align}

  \begin{align}\label{eq3.11}
 \nonumber |I_{10}|\leq&C\int_{Q_{\omega_2}}|B|s^9\lambda^{10}\xi^9|\psi||\varphi|e^{2s\alpha}\,dxdt\\
  \leq&\frac{1}{14}\int_Qs^7\lambda^8\xi^7|\varphi|^2\theta_1e^{2s\alpha}\,dxdt+C\|B\|_{L^{\infty}(Q)}^2\int_{Q_{\omega_2}}s^{12}\lambda^{12}\xi^{12}|\psi|^2e^{2s\alpha}\,dxdt,  
   \end{align}

  \begin{align}\label{eq3.12}
  \nonumber&|I_{11}|+|I_{12}|\leq C\int_{Q_{\omega_2}}|a_1|(s^8\lambda^9\xi^8|\nabla\varphi||\psi|+s^9\lambda^{10}\xi^9|\varphi||\psi|)e^{2s\alpha}\,dxdt \\
\nonumber\leq&C\|a_1\|_{L^{\infty}(Q)}\left(\int_Q\left(s^6\lambda^8\xi^6|\varphi|^2+s^4\lambda^6\xi^4|\nabla\varphi|^2\right)e^{2s\alpha}\,dxdt\right)^{\frac{1}{2}}\left(\int_{Q_{\omega_2}}s^{12}\lambda^{12}\xi^{12}|\psi|^2e^{2s\alpha}\,dxdt\right)^{\frac{1}{2}}\\
  \leq&\frac{1}{14}\int_Qs^7\lambda^8\xi^7|\varphi|^2\theta_1e^{2s\alpha}\,dxdt+C\|a_1\|_{L^{\infty}(Q)}^2\int_{Q_{\omega_2}}s^{12}\lambda^{12}\xi^{12}|\psi|^2e^{2s\alpha}\,dxdt.  \end{align}
Therefore, we deduce from inequalities \eqref{eq3.6}-\eqref{eq3.12} that
\begin{align}\label{eq3.13}
\int_0^T\int_{\mathcal{O}} s^7\lambda^8\xi^7|\varphi|^2e^{2s\alpha}\theta_1\,dxdt
\leq C\int_{Q_{\omega_2}}s^{16}\lambda^{16}\xi^{16}|\psi|^2e^{2s\alpha}\,dxdt
\end{align}  
for any $\lambda\geq\hat{\lambda}(1+\|a_0\|_{L^{\infty}(Q)}^{\frac{1}{4}}+\|B_0\|_{L^{\infty}(Q)}^{\frac{1}{3}}+\|B\|_{L^{\infty}(Q)}^{\frac{1}{2}}+\|a_1\|_{L^{\infty}(Q)}^{\frac{1}{2}})$ and any $s\geq s_0(1+\sqrt{T}+T).$ 

Thus, in view of inequality \eqref{eq3.2} and inequality \eqref{eq3.13}, yields
\begin{align}\label{eq3.14}
\nonumber\int_Qe^{2s\alpha}s^6\xi^6|\varphi|^2\leq &C_1\int_{Q_{\omega_1}}s^7\xi^7|\varphi|^2e^{2s\alpha}\,dxdt\\
\nonumber\leq&\int_0^T\int_{\mathcal{O}} s^7\xi^7|\varphi|^2e^{2s\alpha}\theta_1\,dxdt\\
\leq&C\int_{Q_{\omega_2}}s^{16}\xi^{16}|\psi|^2e^{2s\alpha}\,dxdt
\end{align}
for any fixed $\lambda\geq\hat{\lambda}(1+\|a_0\|_{L^{\infty}(Q)}^{\frac{1}{4}}+\|B_0\|_{L^{\infty}(Q)}^{\frac{1}{3}}+\|B\|_{L^{\infty}(Q)}^{\frac{1}{2}}+\|a_1\|_{L^{\infty}(Q)}^{\frac{1}{2}})$ and any $s\geq s_0(1+\sqrt{T}+T).$

Combining inequalities \eqref{eq3.3} with inequality \eqref{eq3.14}, we obtain
\begin{align}\label{eq3.15}
\int_Qs^6\xi^6|\psi|^2e^{2s\alpha}\,dxdt
\leq C\int_{Q_{\omega_2}}s^{16}\xi^{16}|\psi|^2e^{2s\alpha}\,dxdt
\end{align}
for any fixed $\lambda\geq\hat{\lambda}(1+\|a_0\|_{L^{\infty}(Q)}^{\frac{1}{4}}+\|B_0\|_{L^{\infty}(Q)}^{\frac{1}{3}}+\|B\|_{L^{\infty}(Q)}^{\frac{1}{2}}+\|a_1\|_{L^{\infty}(Q)}^{\frac{1}{2}})$ and any $s\geq s_0(1+\sqrt{T}+T).$ 

Finally, we will combining energy estimates with inequalities \eqref{eq3.14}-\eqref{eq3.15} to obtain the desired observability inequality. At this point, applying classical estimates of the fourth order parabolic equation to systems \eqref{eq3.1}, we obtain for any $t_1,$ $t_2\in [0,T]$ with $t_1<t_2$ and any $t\in[0,T],$
\begin{align}\label{eq3.16}
\|\varphi(t_2)\|_{L^2(D)}^2\leq e^{2\beta(t_2-t_1)}\|\varphi(t_1)\|_{L^2(D)}^2
\end{align}
and
\begin{align}\label{eq3.17}
\|\psi(t)\|_{L^2(D)}^2\leq \int_t^Te^{2\beta(s-t)}\|\varphi(s)\|_{L^2(\mathcal{O})}^2\,ds,
\end{align}
where
\begin{align*}
\beta=2+\|a_0\|_{L^{\infty}(Q)}^2+\|B_0\|_{L^{\infty}(Q)}^2+\|B\|_{L^{\infty}(Q)}^2+\|a_1\|_{L^{\infty}(Q)}^2.
\end{align*}
In particular, we have
\begin{align*}
\|\varphi(t+\frac{T}{4})\|_{L^2(D)}^2\leq e^{\frac{\beta T}{2}}\|\varphi(t)\|_{L^2(D)}^2
\end{align*}
for any $t\in [\frac{T}{4},\frac{3T}{4}],$ which implies that
\begin{align}\label{eq3.18}
\int_{\frac{T}{2}}^T\|\varphi(t)\|_{L^2(D)}^2\,dt\leq e^{\frac{\beta T}{2}}\int_{\frac{T}{4}}^{\frac{3T}{4}}\|\varphi(t)\|_{L^2(D)}^2.
\end{align}
On the other hand, we deduce from inequality \eqref{eq3.17} that
\begin{align*}
\int_t^T\|\psi(s)\|_{L^2(D)}^2\,ds\leq (T-t)e^{\beta T}\int_t^T\|\varphi(s)\|_{L^2(\mathcal{O})}^2\,ds
\end{align*}
for any $t\in [\frac{T}{2},T],$ which entails that
\begin{align}\label{eq3.19}
\int_{\frac{T}{2}}^T\|\psi(s)\|_{L^2(D)}^2\,ds\leq e^{(1+\beta)T}\int_{\frac{T}{2}}^T\|\varphi(s)\|_{L^2(\mathcal{O})}^2\,ds.
\end{align}
Denote by $m_0=\min\limits_{x\in\overline{D}}\alpha_0(x)$ and $M_0=\max\limits_{x\in\overline{D}}\alpha_0(x),$ we deduce from Lemma \ref{le2.3} that for any $s\geq 0,$
\begin{align}\label{eq3.20}
\int_Qs^6\xi^6|\psi|^2e^{2s\alpha}\,dxdt\geq A_s\int_0^{\frac{T}{2}}\int_De^{-\frac{M_s}{\sqrt{t}}}|\psi|^2\,dxdt
\end{align}
with $A_s$ and $M_s$ given in Lemma \ref{le2.3}.  

In what follows, we will Bounding the right hand side of inequality \eqref{eq3.15} by using Lemma \ref{le2.3}, we obtain
\begin{align}\label{eq3.21}
\nonumber\int_Qs^6\xi^6|\psi|^2e^{2s\alpha}\,dxdt\leq &C\int_{Q_{\omega_2}}s^{16}\xi^{16}|\psi|^2e^{2s\alpha}\,dxdt\\
\leq&C2^{48}\left(\frac{N_0}{M_0e}\right)^{16}\int_{Q_{\omega_2}}|\psi|^2\,dxdt
\end{align}
for any $s\geq\frac{4T}{|M_0|}.$ Thus, we obtain
\begin{align}\label{eq3.22}
\int_0^{\frac{T}{2}}\int_De^{-\frac{M_s}{\sqrt{t}}}|\psi|^2\,dxdt
\leq C\frac{2^{30}}{e^{16}}e^{\frac{2|m_0|s}{T}}\left(\frac{N_0^{16}}{M_0^{10}n_0^6}\right)\int_{Q_{\omega_2}}|\psi|^2\,dxdt.
\end{align}
Along with inequality \eqref{eq3.19} and inequality \eqref{eq3.22}, yields
\begin{align}\label{eq3.23}
\nonumber\int_Qe^{-\frac{M_s}{\sqrt{t}}}|\psi|^2\,dxdt\leq&\int_0^{\frac{T}{2}}\int_De^{-\frac{M_s}{\sqrt{t}}}|\psi|^2\,dxdt+\int_{\frac{T}{2}}^T\int_D|\psi|^2\,dxdt\\
\nonumber\leq&C2^{48}\left(\frac{N_0}{M_0e}\right)^{16}\int_{Q_{\omega_2}}|\psi|^2\,dxdt\\
&+e^{(1+\beta)T}\int_{\frac{T}{2}}^T\|\varphi(s)\|_{L^2(\mathcal{O})}^2\,ds.
\end{align}
In view of inequalities \eqref{eq3.15}, \eqref{eq3.18} and Lemma \ref{le2.3}, yields
\begin{align}\label{eq3.24}
\nonumber\int_{\frac{T}{2}}^T\|\varphi(s)\|_{L^2(D)}^2\,ds\leq&e^{\frac{\beta T}{2}}\int_{\frac{T}{4}}^{\frac{3T}{4}}\|\varphi(t)\|_{L^2(D)}^2\\
\nonumber\leq&e^{\frac{\beta T}{2}+\frac{8|m_0|s}{\sqrt{3}T}}(2n_0)^{-6}T^6\int_{\frac{T}{4}}^{\frac{3T}{4}}\int_D\xi^6e^{2s\alpha}|\varphi(t)|^2\,dxdt\\
\nonumber\leq&Ce^{\frac{\beta T}{2}+\frac{8|m_0|s}{\sqrt{3}T}}(2n_0)^{-6}T^6\int_{\frac{T}{4}}^{\frac{3T}{4}}\int_{Q_{\omega_2}}s^{16}\xi^{16}|\psi|^2e^{2s\alpha}\,dxdt\\
\leq&C2^{42}\left(\frac{N_0}{M_0}\right)^{16}e^{\frac{\beta T}{2}+\frac{8|m_0|s}{\sqrt{3}T}-16}n_0^{-6}T^6\int_{\frac{T}{4}}^{\frac{3T}{4}}\int_{Q_{\omega_2}}|\psi|^2\,dxdt
\end{align}
for any $s\geq \frac{4T}{|M_0|}.$

Combining inequality \eqref{eq3.23} with inequality \eqref{eq3.24}, we obtain
\begin{align}\label{eq3.25}
\int_Qe^{-\frac{M_s}{\sqrt{t}}}|\psi|^2\,dxdt\leq H\int_{Q_{\omega}}|\psi|^2\,dxdt
\end{align}
for any $s\geq \frac{4T}{|M_0|},$ where 
\begin{align*}
H=C2^{48}\left(\frac{N_0}{M_0e}\right)^{16}+C2^{42}\left(\frac{N_0}{M_0e}\right)^{16}e^{2\beta T+\frac{8|m_0|s}{\sqrt{3}T}}n_0^{-6}T^6.
\end{align*}
\end{proof}
In the following, we will prove the existence of an insensitizing control such that the solution of problem \eqref{eq1.1} verifies condition \eqref{eq1.3}, i.e., we will prove the null-controllability of problem \eqref{eq1.4}.
%\begin{theorem}\label{th3.2}
%Assume that $\omega\cap \mathcal{O}\neq \emptyset$ and $y_0=0.$ Then there exists a positive constant $M$ depending on $D,$ $\omega,$ $\mathcal{O},$ $T$ such that for any $f\in L^2(Q)$ satisfying
%\begin{align*}
%\int_Q\exp\left(\frac{M}{\sqrt{t}}\right)|f|^2\,dxdt<+\infty,
%\end{align*}
%one can find a control function $v\in L^2(Q_\omega)$ insensitizing the functional $\Phi$ given by \eqref{eq1.2}.
%\end{theorem}
%
%
%
%
%
%This section is devoted to the proof of Theorem \ref{1.4}. We start with the existence of approximately insensitizing controls for the following problems \eqref{1.5}-\eqref{1.6} with $y_0=0.$ Thus, we consider the following problems
%\begin{equation}\label{3.1}
%\begin{cases}
%\frac{\partial y}{\partial t}+\Delta^2 y+a_0y+B_0\cdot\nabla y+B:\nabla^2 y+a_1\Delta y=\chi_\omega v+f,\,\,\,\,\textit{in}\,\,\,Q,\\
%y=\Delta y=0,\,\,\,\,\,\textit{on}\,\,\,\,\Sigma,\\
%y(x,0)=0,\,\,\,\,\textit{in}\,\,\,\,D,
%\end{cases}
%\end{equation}
%and
%\begin{equation}\label{3.2}
%\begin{cases}
%-\frac{\partial q}{\partial t}+\Delta^2 q+a_0q-\nabla\cdot(B_0 q)+\sum_{i,j=1}^n\frac{\partial^2(B_{ij}y)}{\partial x_i\partial x_j}+\Delta(a_1y)=\chi_{\mathcal{O}} y,\,\,\,\,\textit{in}\,\,\,Q,\\
%q=\Delta q=0,\,\,\,\,\,\textit{on}\,\,\,\,\Sigma,\\
%q(x,T)=0,\,\,\,\,\textit{in}\,\,\,\,D,
%\end{cases}
%\end{equation}
%and the corresponding adjoint systems \eqref{2.6}-\eqref{2.7}. The result is stated as follows.
\begin{theorem}\label{th3.2}
Assume that $\omega\cap \mathcal{O}\neq \emptyset,$ $y_0=0$ and the positive constants $M$ and $H$ are defined as in Theorem \ref{th3.1}. If $f\in L^2(Q)$ satisfies
\begin{align*}
\int_Qe^{\frac{M}{\sqrt{t}}}|f|^2\,dxdt<+\infty,
\end{align*}
then there exists a control $v\in L^2(Q_\omega),$ such that the solution $(y,q)$ of problem \eqref{eq1.4} satisfies
\begin{align}\label{eq3.26}
q(x,0)\equiv 0,\,\,\forall\,\,x\in D.
\end{align}
Moreover, we also have
\begin{align*}
\|v\|_{L^2(Q_\omega)}\leq2\sqrt{H}\left(\int_Q e^{\frac{M}{\sqrt{t}}}|f|^2\,dxdt\right)^{\frac{1}{2}}.
\end{align*}
%
%For any $\epsilon>0,$ there exists a control function $v_\epsilon\in L^2(Q_\omega)$ such that the associated solution $(y_\epsilon,q_\epsilon)$ of problem \eqref{eq1.4} satisfies
%\begin{align}\label{3.4}
%\|q_\epsilon(0)\|_{L^2(D)}\leq \epsilon.
%\end{align}
%In addition, if $f\in L^2(Q)$ satisfies
%\begin{align*}
%\int_Qe^{\frac{M}{\sqrt{t}}}|f|^2\,dxdt<+\infty,
%\end{align*}
%then the control $\{v_\epsilon\}_{\epsilon>0}$ are uniformly bounded in $L^2(Q_\omega).$ More precisely,
%\begin{align*}
%\|v_\epsilon\|_{L^2(Q_\omega)}\leq2\sqrt{H}\left(\int_Q e^{\frac{M}{\sqrt{t}}}|f|^2\,dxdt\right)^{\frac{1}{2}},\,\,\,\,\forall\,\,\,\epsilon>0.
%\end{align*}
\end{theorem}
\begin{proof}
In what follows, we will prove the null controllability of problem \eqref{eq1.1} by the similar method in \cite{fc}. To this purpose, for any $\epsilon>0,$ we introduce a functional defined on $L^2(D):$
\begin{align*}
\mathcal{J}(\varphi_0)=\frac{1}{2}\int_{Q_\omega}|\psi|^2\,dxdt+\epsilon\|\varphi_0\|_{L^2(D)}+\int_Q f\psi\,dxdt,
\end{align*}
where $(\psi,\varphi)$ is the solution of problem \eqref{eq3.1} with initial data $\psi(0)=0$ and 
$\varphi(0)=\varphi_0\in L^2(D).$

In view of Theorem \ref{th3.1},  we conclude that the functional $\mathcal{J}(\varphi_0)$ is continous, strictly convex and coercive on $L^2(D).$ Therefore, for any $\epsilon>0,$ there exists a unique minimum point $\varphi_{0\epsilon}\in L^2(D)$ of $\mathcal{J},$ which implies that
\begin{align}\label{eq3.27}
0=\mathcal{J}(0)\geq\mathcal{J}(\varphi_{0\epsilon})=\frac{1}{2}\int_{Q_\omega}|\psi_\epsilon|^2\,dxdt+\epsilon\|\varphi_{0\epsilon}\|_{L^2(D)}+\int_Q f\psi_\epsilon\,dxdt,
\end{align}
where $(\psi_\epsilon,\varphi_\epsilon)$ solves problem \eqref{eq3.1} with initial data $(0,\varphi_{0\epsilon}).$

Therefore, we deduce from inequalities \eqref{eq3.0}, \eqref{eq3.27} and H\"{o}lder's inequality that
\begin{align*}
\frac{1}{2}\int_{Q_\omega}|\psi_\epsilon|^2\,dxdt+\epsilon\|\varphi_{0\epsilon}\|_{L^2(D)}\leq&-\int_Q f\psi_\epsilon\,dxdt\\
\leq&\left(\int_Qe^{\frac{M}{\sqrt{t}}}|f|^2\,dxdt\right)^{\frac{1}{2}}\left(\int_Qe^{-\frac{M}{\sqrt{t}}}|\psi_\epsilon|^2\,dxdt\right)^{\frac{1}{2}}\\
\leq&\sqrt{H}\left(\int_Qe^{\frac{M}{\sqrt{t}}}|f|^2\,dxdt\right)^{\frac{1}{2}}\left(\int_{Q_\omega}|\psi_\epsilon|^2\,dxdt\right)^{\frac{1}{2}}
\end{align*}
for any $s\geq\frac{4T}{|M_0|}.$ 

Employing Young's inequality, yields
\begin{align*}
\int_{Q_\omega}|\psi_\epsilon|^2\,dxdt+4\epsilon\|\varphi_{0\epsilon}\|_{L^2(D)}
\leq 4H\int_Qe^{\frac{M}{\sqrt{t}}}|f|^2\,dxdt
\end{align*}
for any $s\geq\frac{4T}{|M_0|}.$

If $\varphi_{0\epsilon}\neq 0,$ then $\mathcal{J}$ satisfies the optimality condition
\begin{align}\label{eq3.28}
\int_{Q_\omega}\psi_\epsilon\psi\,dxdt+\int_Qf\psi\,dxdt+\frac{\epsilon}{\|\varphi_{0\epsilon}\|_{L^2(D)}}\int_D\varphi_{0\epsilon}\varphi_0\,dx=0
\end{align}
for any $\varphi_0\in L^2(D),$ where $(\psi,\varphi)$ is the solution of problem \eqref{eq3.1} with initial data $(0,\varphi_0).$

Now, let $v_\epsilon=\psi_\epsilon$ and let $(y_\epsilon,q_\epsilon)$ be the solution of problem \eqref{eq1.4}, then we infer from problem \eqref{eq1.4} and problem \eqref{eq3.1} that
\begin{align}\label{eq3.29}
\int_Q\chi_{\mathcal{O}}\varphi y_\epsilon\,dxdt=\int_Q(\chi_\omega v_\epsilon+f)\psi\,dxdt
\end{align}
and 
\begin{align}\label{eq3.30}
\int_Dq_\epsilon(x,0)\varphi_0\,dx=\int_Q\chi_{\mathcal{O}}\varphi y_\epsilon\,dxdt.
\end{align}
Thus, along with inequalities \eqref{eq3.28}-\eqref{eq3.30} and the fact that $v_\epsilon=\psi_\epsilon,$ we obtain
\begin{align}\label{eq3.31}
\int_Dq_\epsilon(x,0)\varphi_0\,dx=-\frac{\epsilon}{\|\varphi_{0\epsilon}\|_{L^2(D)}}\int_D\varphi_{0\epsilon}\varphi_0\,dx
\end{align}
for any $\varphi_0\in L^2(D),$ which implies that
\begin{align}\label{eq3.32}
\|q_\epsilon(0)\|_{L^2(D)}\leq \epsilon.
\end{align}
If $\varphi_{0\epsilon}=0,$ then 
\begin{align*}
\lim_{t\rightarrow 0}\frac{\mathcal{J}(t\varphi_0)}{t}\geq 0
\end{align*}
for any $\varphi_0\in L^2(D),$ i.e., 
\begin{align}\label{eq3.33}
\epsilon\|\varphi_0\|_{L^2(D)}+\int_Q f\psi\,dxdt\geq 0,
\end{align}
where $(\psi,\varphi)$ solves problem \eqref{eq3.1} with initial data $(0,\varphi_0).$ Consequently, we can also conclude from inequalities \eqref{eq3.29}-\eqref{eq3.30}, \eqref{eq3.33} and the fact that $v_\epsilon=\psi_\epsilon=0$ that
\begin{align*}
\epsilon\|\varphi_0\|_{L^2(D)}+\int_Dq_\epsilon(x,0)\varphi_0\,dx\geq 0
\end{align*}
for any $\varphi_0\in L^2(D),$ which also implies that
\begin{align*}
\|q_\epsilon(0)\|_{L^2(D)}\leq \epsilon.
\end{align*}
Therefore, the solution $(y_\epsilon,q_\epsilon)$ of problem \eqref{eq1.4} associated with $v_\epsilon$ satisfies inequality 
\begin{align}\label{eq3.35}
\|q_\epsilon(0)\|_{L^2(D)}\leq \epsilon.
\end{align}
Moreover, we obtain
\begin{align*}
\int_{Q_\omega}|v_\epsilon|^2\,dxdt\leq 4H\int_Qe^{\frac{M}{\sqrt{t}}}|f|^2\,dxdt
\end{align*}
for any $s\geq\frac{4T}{|M_0|},$ which entails that the controls $\{v_\epsilon\}_{\epsilon>0}$ are uniformly bounded in $L^2(Q_\omega).$ Without loss of generality, we can assume that $v_\epsilon\rightharpoonup v$ weakly in $L^2(Q_\omega)$ and 
\begin{align*}
(y_\epsilon,q_\epsilon)\rightharpoonup (y,q),\,\,\,\,\,\textit{weakly\,\,\,in}\,\,\,X\times X,
\end{align*}
where $(y,q)$ is the solution of problem \eqref{eq1.4} with $v.$ In particular, we have the weak convergence of $q_\epsilon(0)$ in $L^2(D).$ Thus, we conclude from inequality \eqref{eq3.35} that  $q(0)\equiv 0,$ i.e., $v$ is the desired control. Moreover, we have
\begin{align*}
\int_{Q_\omega}|v_\epsilon|^2\,dxdt\leq 4H\int_Qe^{\frac{M}{\sqrt{t}}}|f|^2\,dxdt
\end{align*}
for any $s\geq\frac{4T}{|M_0|}.$
\end{proof}
\section{The semi-linear case}
\def\theequation{4.\arabic{equation}}\makeatother
\setcounter{equation}{0}
In this section, under the assumptions that $F\in W^{1,\infty}(\mathbb{R}\times\mathbb{R}^n\times\mathbb{R}^{n^2};\mathbb{R})$ and $y_0=0,$ we will prove the existence of an insensitizing control of problem 
\begin{equation}\label{eq4.1}
\begin{cases}
\frac{\partial y}{\partial t}+\Delta^2 y+a_0y+B_0\cdot\nabla y+B:\nabla^2 y+a_1\Delta y=F(y,\nabla y,\nabla^2y)+v\chi_\omega+f,\,\,\,\,\forall\,\,\,(x,t)\in Q,\\
-\frac{\partial q}{\partial t}+\Delta^2 q+a_0q-\nabla\cdot(B_0 q)+\sum_{i,j=1}^n\frac{\partial^2(B_{ij}q)}{\partial x_i\partial x_j}+\Delta(a_1q)=F_y(y,\nabla y,\nabla^2y)q\\
-\nabla\cdot(\nabla_pF(y,\nabla y,\nabla^2y)q)+\sum_{i,j=1}^n\frac{\partial^2(F_{r_{ij}}(y,\nabla y,\nabla^2y)q)}{\partial x_i\partial x_j}+\chi_{\mathcal{O}} y,\,\,\,\,\forall\,\,\,(x,t)\in Q,\\
y=\Delta y=0,\,\,\,q=\Delta q=0,\,\,\forall\,\,\,\,(x,t)\in\Sigma,\\
y(x,0)=0,\,\,q(x,T)=0,\,\,\forall\,\,\,\,x\in D
\end{cases}
\end{equation}
such that
\begin{align}\label{eq4.2}
q(x,0)\equiv 0,\,\,\,\forall x\in D.
\end{align}
From the regularity of fourth order parabolic equations, we conclude that there exists a unique solution of problem \eqref{eq4.1} satisfying
\begin{align*}
&y\in Y=L^2(0,T; H_0^1(D)\cap H^4(D))\cap H^1(0,T;L^2(D)),\\
&q\in X=L^2(0,T; H_0^1(D)\cap H^2(D))\cap H^1(0,T;(H^2(D))^*).
\end{align*}
In what follows, we will establish the existence of an insensitizing  control such that the solution of problem \eqref{eq4.1} verifying \eqref{eq4.2} in the semi-linear case.
\begin{theorem}\label{th4.1}
Assume that $\omega\cap \mathcal{O}\neq \emptyset,$ $y_0=0,$ $F\in W^{1,\infty}(\mathbb{R}\times\mathbb{R}^n\times\mathbb{R}^{n^2};\mathbb{R}),$ the assumption on $f$ is given as in Theorem \ref{th3.2}. Then there exists a control $v\in L^2(Q_\omega),$ such that the solution $(y,q)$ of problem \eqref{eq4.1} satisfies \eqref{eq4.2}.
\end{theorem}

\begin{proof}
Let $z\in L^2(0,T;H_0^1(D)\cap H^2(D))$ be given, consider the following problem 
\begin{equation}\label{eq4.3}
\begin{cases}
\frac{\partial y}{\partial t}+\Delta^2y+a_0y+B_0\cdot\nabla y+B:\nabla^2y+a_1\Delta y=G_1(z,\nabla z,\nabla^2z)y+G_2(z,\nabla z,\nabla^2z)\cdot \nabla y\\
+G_3(z,\nabla z,\nabla^2z): \nabla^2 y+F(0,0,0)+v\chi_\omega+f,\,\,\,\,(x,t)\in Q,\\
-\frac{\partial q}{\partial t}+\Delta^2q+a_0q-\nabla\cdot(B_0q)+\sum_{i,j=1}^n\frac{\partial^2(B_{ij}q)}{\partial x_i\partial x_j}+\Delta(a_1q)=F_y(z,\nabla z,\nabla^2z)q\\
-\nabla\cdot(\nabla_pF(z,\nabla z,\nabla^2z)q)+\sum_{i,j=1}^n\frac{\partial^2(F_{r_{ij}}(z,\nabla z,\nabla^2z)q)}{\partial x_i\partial x_j}+y\chi_{\mathcal{O}},\,\,\,\,(x,t)\in Q,\\ 
y=\Delta y=0,\,\,\,q=\Delta q=0,\,\,(x,t)\in\Sigma,\\
y(x,0)=0,\,\,q(x,T)=0,\,\,\,x\in D,
\end{cases}
\end{equation}
where
\begin{align*}
&G_1(w,\nabla w,\nabla^2 w)=\int_0^1\frac{\partial F}{\partial y}(\tau w,\tau\nabla w,\tau\nabla^2w)\,d\tau,\\
&G_2(w,\nabla w,\nabla^2w)=\int_0^1\nabla_pF(\tau w,\tau\nabla w,\tau\nabla^2w)\,d\tau,\\
&G_3^{ij}(w,\nabla w,\nabla^2w)=\int_0^1\frac{\partial F}{\partial r_{ij}}(\tau w,\tau\nabla w,\tau\nabla^2w)\,d\tau.
\end{align*}
Since $F\in W^{1,\infty}(\mathbb{R}\times\mathbb{R}^n\times\mathbb{R}^{n^2},\mathbb{R}),$ there exists a positive constant $M,$ such that
\begin{align*}
|G_1(u,p,r)|+|G_2(u,p,r)|+|G_3(u,p,r)|\leq M,\,\,\,\,\,\forall\,\,\,\,(u,p,r)\in\mathbb{R}\times\mathbb{R}^n\times\mathbb{R}^{n^2}
\end{align*}
and
\begin{align*}
|F_y(u,p,r)|+|\nabla_pF(u,p,r)|+\sum_{i,j=1}^n\left|\frac{\partial F}{\partial r_{ij}}(u,p,r)\right|\leq M,\,\,\,\,\,\forall\,\,\,\,(u,p,r)\in\mathbb{R}\times\mathbb{R}^n\times\mathbb{R}^{n^2}.
\end{align*}

From Theorem \ref{th3.1}, we conclude that there exists at least one control $v\in L^2(Q_\omega),$ such that the solution $(y^z,q^z)$ of problem \eqref{eq4.3} satisfies
\begin{align}\label{eq4.4}
q^z(x,0)\equiv 0,\,\,\forall\,\,x\in D.
\end{align}
Moreover, we also have
\begin{align}\label{eq4.5}
\|v^z\|_{L^2(Q_\omega)}\leq2\sqrt{H}\left(\int_Q e^{\frac{M}{\sqrt{t}}}|f|^2\,dxdt\right)^{\frac{1}{2}}.
\end{align}
In what follows, we denote by $v^z$ the control with the minimal $L^2(Q_\omega)$-norm in the set of 
the controls such that the solution $(y^z,q^z)$ of problem \eqref{eq4.3} corresponding to $z$ satisfies \eqref{eq4.4}.

From the regularity theory of parabolic equations, we conclude that there exists a unique weak solution $(y^z,q^z)\in Y\times X.$  Moreover, since $F\in W^{1,\infty}(\mathbb{R}\times\mathbb{R}^n\times\mathbb{R}^{n^2};\mathbb{R}),$  there exists a positive constant $C$ independent of $z,$ such that
\begin{align}\label{eq4.6}
\nonumber\|y^z\|_Y+\|q^z\|_X\leq  &C(\|F(0,0,0)+v^z\chi_\omega+f\|_{L^2(Q)})\\
\leq  &C(1+\|v^z\|_{L^2(Q_\omega)}+\|f\|_{L^2(Q)}).
\end{align}
Thus, along with inequalities \eqref{eq4.5}-\eqref{eq4.6}, we deduce that there exists a positive constant $\mathcal{L}_1$ independent of $z,$ such that 
\begin{align}\label{eq4.7}
\|y^z\|_Y+\|q^z\|_X\leq  \mathcal{L}_1\left(1+\|e^{\frac{M}{2\sqrt{t}}}f\|_{L^2(Q)}\right).
\end{align}
Define $\Lambda:L^2(0,T;H_0^1(D)\cap H^2(D))\rightarrow L^2(0,T;H_0^1(D)\cap H^2(D))$ by
\begin{align*}
\Lambda(z)=y^z,
\end{align*}
then the mapping $\Lambda$ is well-defined. In what follows, we will prove the existence of a fixed point for the operator $\Lambda$ by the Leray-Schauder's fixed points Theorem. To this purpose, we will first prove that $\Lambda$ is continuous, i.e., if $z_k\rightarrow z$ in $L^2(0,T;H_0^1(D)\cap H^2(D)),$ we have $\Lambda(z_k)\rightarrow \Lambda(z).$

Denote by $y^k=\Lambda(z_k),$ where $(y^k,q^k)$ is the solution of problem
\begin{equation}\label{eq4.8}
\begin{cases}
\frac{\partial y^k}{\partial t}+\Delta^2y^k+a_0y^k+B_0\cdot\nabla y^k+B:\nabla^2y^k+a_1\Delta y^k=G_1(z_k,\nabla z_k,\nabla^2 z_k)y^k\\
+G_2(z_k,\nabla z_k,\nabla^2 z_k)\cdot \nabla y^k+G_3(z_k,\nabla z_k,\nabla^2 z_k): \nabla^2 y^k+F(0,0,0)+v_{z_k}\chi_\omega+f,\,\,\,\,(x,t)\in Q,\\
-\frac{\partial q^k}{\partial t}+\Delta^2q^k+a_0q^k-\nabla\cdot(B_0q^k)+\sum_{i,j=1}^n\frac{\partial^2(B_{ij}q^k)}{\partial x_i\partial x_j}+\Delta(a_1q^k)=F_y(z_k,\nabla z_k,\nabla^2 z_k)q^k\\
-\nabla\cdot(\nabla_pF(z_k,\nabla z_k,\nabla^2 z_k)q^k)+\sum_{i,j=1}^n\frac{\partial^2(F_{r_{ij}}(z_k,\nabla z_k,\nabla^2z_k)q^k)}{\partial x_i\partial x_j}+y^k\chi_{\mathcal{O}},\,\,\,\,(x,t)\in Q,\\ 
y^k=\Delta y^k=0,\,\,\,q^k=\Delta q^k=0,\,\,(x,t)\in\Sigma,\\
y^k(x,0)=0,\,\,q^k(x,T)=0,\,\,\,x\in D.
\end{cases}
\end{equation}
It follows from inequality \eqref{eq4.7} and the fact that $z_k\rightarrow z$ in $L^2(0,T;H_0^1(D)\cap H^2(D))$ that
\begin{align*}
&\{(y^k,q^k)\}_{k=1}^{\infty}\,\,\,\textit{is\,\,\,uniformly\,\,\,bounded\,\,in}\,\,Y\times X,\\
&\{v^{z_k}\}_{k=1}^{\infty}\,\,\,\textit{is\,\,\,uniformly\,\,\,bounded\,\,in}\,\,L^2(Q_\omega),
\end{align*}
which entails that there exists a subsequence of $\{y^k\}_{k=1}^\infty,$ $\{q^k\}_{k=1}^\infty,$ $\{v^{z_k}\}_{k=1}^\infty$ (still denote by themselves) and $y\in Y,$ $q\in X,$ $v\in L^2(Q_\omega),$ such that
\begin{align*}
&y^k\rightharpoonup y\,\,\,\textit{in}\,\,Y\,\,\,\textit{as}\,\,k\rightarrow+\infty,\\
&y^k\rightarrow y\,\,\,\textit{in}\,\,L^2(0,T;H_0^1(D)\cap H^2(D))\,\,\,\textit{as}\,\,k\rightarrow+\infty,\\
&q^k\rightharpoonup q\,\,\,\textit{in}\,\,X\,\,\,\textit{as}\,\,k\rightarrow+\infty,\\
&v^{z_k}\rightharpoonup v\,\,\,\textit{in}\,\,L^2(Q_\omega)\,\,\,\textit{as}\,\,k\rightarrow+\infty.
\end{align*}
Since $F\in W^{1,\infty}(\mathbb{R}\times\mathbb{R}^n\times\mathbb{R}^{n^2},\mathbb{R}),$ we conclude that there exists a subsequence of $\{G_1(z_k,\nabla z_k,\nabla^2 z_k)\}_{k=1}^\infty,$ $\{G_2(z_k,\nabla z_k,\nabla^2 z_k)\}_{k=1}^\infty,$ $\{G_3(z_k,\nabla z_k,\nabla^2 z_k)\}_{k=1}^\infty,$ $\{F_y(z_k,\nabla z_k,\nabla^2z_k)\}_{k=1}^\infty,$ $\{\nabla_pF(z_k,\nabla z_k,\nabla^2z_k)\}_{k=1}^\infty,$ $\{(F_{r_{ij}}(z_k,\nabla z_k,\nabla^2z_k))_{1\leq i, j\leq n}\}_{k=1}^\infty,$ (still denote by themselves), such that
\begin{align*}
&G_1(z_k,\nabla z_k,\nabla^2z_k)\rightarrow G_1(z,\nabla z,\nabla^2 z)\,\,\,\textit{weakly\,\,star\,\,in}\,\,L^{\infty}(Q),\,\,\textit{as}\,\,k\rightarrow+\infty,\\
&G_2(z_k,\nabla z_k,\nabla^2z_k)\rightarrow G_2(z,\nabla z,\nabla^2z)\,\,\,\textit{weakly\,\,star\,\,in}\,\,L^{\infty}(Q),\,\,\textit{as}\,\,k\rightarrow+\infty,\\
&G_3(z_k,\nabla z_k,\nabla^2z_k)\rightarrow G_3(z,\nabla z,\nabla^2z)\,\,\,\textit{weakly\,\,star\,\,in}\,\,L^{\infty}(Q),\,\,\textit{as}\,\,k\rightarrow+\infty,\\
&F_y(z_k,\nabla z_k,\nabla^2z_k)\rightarrow F_y(z,\nabla z,\nabla^2z)\,\,\,\textit{weakly\,\,star\,\,in}\,\,L^{\infty}(Q),\,\,\textit{as}\,\,k\rightarrow+\infty,\\
&\nabla_pF(z_k,\nabla z_k,\nabla^2z_k)\rightarrow\nabla_pF(z,\nabla z,\nabla^2z)\,\,\,\textit{weakly\,\,star\,\,in}\,\,L^{\infty}(Q),\,\,\textit{as}\,\,k\rightarrow+\infty,\\
&F_{r_{ij}}(z_k,\nabla z_k,\nabla^2z_k)\rightarrow F_{r_{ij}}(z,\nabla z,\nabla^2z)\,\,\,\textit{weakly\,\,star\,\,in}\,\,L^{\infty}(Q),\,\,\textit{as}\,\,k\rightarrow+\infty.
\end{align*}
Let $k\rightarrow+\infty$ in problem \eqref{eq4.8}, we obtain
\begin{equation}\label{eq4.9}
\begin{cases}
\frac{\partial y}{\partial t}+\Delta^2y+a_0y+B_0\cdot\nabla y+B:\nabla^2y+a_1\Delta y=G_1(z,\nabla z,\nabla^2 z)y+G_2(z,\nabla z,\nabla^2z)\cdot \nabla y\\
+G_3(z,\nabla z,\nabla^2z): \nabla^2 y+F(0,0,0)+v\chi_\omega+f,\,\,\,\,(x,t)\in Q,\\
-\frac{\partial q}{\partial t}+\Delta^2q+a_0q-\nabla\cdot(B_0q)+\sum_{i,j=1}^n\frac{\partial^2(B_{ij}q)}{\partial x_i\partial x_j}+\Delta(a_1q)=F_y(z,\nabla z,\nabla^2z)q\\
-\nabla\cdot(\nabla_pF(z,\nabla z,\nabla^2z)q)+\sum_{i,j=1}^n\frac{\partial^2(F_{r_{ij}}(z,\nabla z,\nabla^2z)q)}{\partial x_i\partial x_j}+y\chi_{\mathcal{O}},\,\,\,\,(x,t)\in Q,\\ 
y=\Delta y=0,\,\,\,q=\Delta q=0,\,\,(x,t)\in\Sigma,\\
y(x,0)=y_0(x),\,\,q(x,T)=0,\,\,\,x\in D
\end{cases}
\end{equation}
and
\begin{align}\label{eq4.10}
q(x,0)\equiv 0,\,\,\forall\,\,x\in D,
\end{align}
which entails that $y=\Lambda(z).$ Thus, we have proved that $\Lambda(z_k)\rightarrow\Lambda(z)$ in $L^2(0,T;H_0^1(D)\cap H^2(D)),$ i.e., the mapping $\Lambda:L^2(0,T;H_0^1(D)\cap H^2(D))\rightarrow L^2(0,T;H_0^1(D)\cap H^2(D))$ is continuous. Thanks to the compactness of $X\subset L^2(0,T;H_0^1(D)\cap H^2(D))$ and inequality \eqref{eq4.7}, we conclude that the mapping $\Lambda:L^2(0,T;H_0^1(D)\cap H^2(D))\rightarrow L^2(0,T;H_0^1(D)\cap H^2(D))$ is compact. 
Denote by 
\begin{align*}
\mathcal{R}_1=\mathcal{L}_1\left(1+\|e^{\frac{M}{2\sqrt{t}}}f\|_{L^2(Q)}\right)
\end{align*}
and
\begin{align*}
B=\{u\in L^2(0,T;H_0^1(D)\cap H^2(D)):\|u\|_{L^2(0,T;H_0^1(D)\cap H^2(D))}\leq\mathcal{R}_1\},
\end{align*}
then $\Lambda: B\rightarrow B.$ Thus, we can employ the Leray-Schauder's fixed points Theorem to conclude that the operator $\Lambda$ possesses at least one fixed point $y\in L^2(0,T;H_0^1(D)\cap H^2(D)).$ That is, for any $y_0\in L^2(\Omega),$ there exist at least one control $v\in L^2(Q_\omega),$ such that the corresponding solutions to problem \eqref{eq4.1} satisfy $q(x,0)\equiv 0$ for any $x\in D$
\end{proof}
\begin{corollary}\label{co4.1}
Assume that $\omega\cap \mathcal{O}\neq \emptyset,$ $y_0=0,$ $F\in W^{1,\infty}(\mathbb{R}\times\mathbb{R}^n\times\mathbb{R}^{n^2};\mathbb{R}),$ the assumption on $f$ is given as in Theorem \ref{th3.2}. Then there exists a control $v\in L^2(Q_\omega)$ insensitizing the functional $\Phi$ defined by \eqref{eq1.2} for problem \eqref{eq1.1}.
\end{corollary}

\begin{remark}
Under the same assumptions as in Corollary \ref{co4.1}. If problem \eqref{eq1.1} is subject to the homogeneous Dirichlet boundary conditions $y|_{\Sigma}=\frac{\partial y}{\partial\vec{n}}|_{\Sigma}=0,$ the same conclusion as in Corollary \ref{co4.1} remains true.
\end{remark}

\section*{Acknowledgement}
This work was supported by the National Science Foundation of China Grant (11801427, 11871389) and the Fundamental Research Funds for the Central Universities (xzy012022008, JB210714).

\bibliographystyle{plain}
\bibliography{BIB}
\end{document}